\documentclass[12pt,a4paper]{amsart}
\usepackage[a4paper, left=28mm, right=28mm, top=28mm, bottom=34mm]{geometry}

\usepackage{latexsym}
\usepackage{amsmath}
\usepackage{amssymb}
\usepackage{amsthm}
\usepackage{amscd}
\usepackage{mathrsfs}
\usepackage{mathtools}
\usepackage[all]{xy}
\usepackage{graphicx}
\usepackage{comment}
\usepackage{arydshln}
\usepackage{stmaryrd}
\usepackage{url} 
\usepackage{textcase}
\usepackage[dvipsnames]{xcolor}
\usepackage[normalem]{ulem}

\newtheorem{theorem}{Theorem}[section]
\newtheorem{lemma}[theorem]{Lemma}
\newtheorem{corollary}[theorem]{Corollary}
\newtheorem{proposition}[theorem]{Proposition}

\theoremstyle{definition}

\newtheorem{remark}[theorem]{Remark}

\def\A{{\mathbb A}}
\def\F{{\mathbb F}}

\def\Z{{\mathbb Z}}

\def\Ql{\overline{\mathbb{Q}}_\ell}

\newcommand{\mc}{\mathcal}
\newcommand{\et}{\text{\'et}}

\theoremstyle{remark}

\makeatletter

\@addtoreset{equation}{section}
\makeatother

\makeatletter

\@addtoreset{equation}{section}
\makeatother

\makeatletter
\newcommand{\subsubsubsection}{\@startsection{paragraph}{4}{\z@}%
 {1.0\Cvs \@plus.5\Cdp \@minus.2\Cdp}%
 {.1\Cvs \@plus.3\Cdp}%
 {\reset@font\sffamily\normalsize}
 }
\makeatother
\DeclareMathOperator{\Spec}{Spec}

\DeclareMathOperator{\disc}{disc}

\DeclareMathOperator{\Hom}{Hom}

\DeclareMathOperator{\Tr}{Tr}
\DeclareMathOperator{\Nr}{Nr}

\usepackage{bm}

\begin{document}

\title[Factorizations of linearized polynomials and extremal Curves]{Factorizations of linearized polynomials and extremal curves in odd characteristic}

\author{Tetsushi Ito}
\address{
Department of Mathematics, Faculty of Science, Kyoto University
Kyoto, 606--8502, Japan}
\email{tetsushi@math.kyoto-u.ac.jp}

\author{Daichi Takeuchi}
\address{Department of Mathematics,
Institute of Science Tokyo,
2-12-1 Ookayama, Meguro-ku, Tokyo, 152-8551, Japan
}
\email{daichi.takeuchi4@gmail.com}

\author{Takahiro Tsushima}
\address{
Keio University School of Medicine,
4-1-1 Hiyoshi, Kohoku-ku,
Yokohama, 223-8521, Japan}
\email{tsushima@keio.jp}

\date{}

\subjclass[2020]{Primary: 14F20, 14G15; Secondary: 14H25, 14G10.}

\keywords{Artin--Schreier curves; extremal van der Geer--van der Vlugt curves; maximal curves}

\begin{abstract}
We give a complete recipe for constructing extremal van der Geer--van der Vlugt curves over finite fields of odd 
characteristic. The input data consist of a nonzero element of the base field together with a linear subspace satisfying a certain trace condition. This construction may be viewed as an odd-characteristic analogue of the one previously obtained by the authors in characteristic two. 
\end{abstract}
\maketitle
\section{Introduction}\label{Introduction}
Let $p_0$ be a prime number, $p$ a power of $p_0$, and $q$ a power of $p$. 
An $\F_p$-linearized polynomial is a polynomial of the form $f(x)=\sum_{i=0}^e b_i x^{p^i}$. 
A \emph{van der Geer--van der Vlugt curve} over $\F_q$ is a smooth compactification of the affine curve 
\[
C_R \colon y^p-y=xR(x), 
\]
where $R(x)\in\F_q[x]$ is a nonzero $\F_p$-linearized polynomial. 
Van der Geer--van der Vlugt curves have been studied from the viewpoints of number theory and coding theory: see \cite{AM, BHMSSV, CO0, GV, ITT, ITT0, ITT1, TT}.
This family of curves is related to the geometry of Lubin--Tate spaces, which play a central role in the local Langlands correspondences for general linear groups; see \cite{IT, We}. 

Recall that a smooth projective curve $C$ over $\F_q$ is 
\textit{$\F_q$-maximal} (resp.\ \textit{$\F_q$-minimal}) if 
\[
\# C(\F_q)=q+1+2 g(C) \sqrt{q}
\qquad \left(\textrm{resp.\ } \# C(\F_q)=q+1-2 g(C) \sqrt{q}\right), 
\]
where $g(C)$ denotes the genus of $C$. Following \cite{ITT1}, 
we say that $C$ is \textit{$\F_q$-extremal} if it is $\F_q$-maximal or $\F_q$-minimal. 

In this paper, we give a complete recipe for constructing $\F_q$-extremal van der Geer--van der Vlugt curves in odd characteristic.
Our construction may be viewed as an odd-characteristic analogue of the construction developed in \cite{ITT1} for the characteristic-two case.

Assume that $p_0$ is odd. A key ingredient in our approach is the notion of the adjoint of an
$\F_p$-linearized polynomial. 
For an $\F_p$-linearized polynomial 
$f(x)=\sum_{i=0}^e b_i x^{p^i} \in \F_q[x]$, 
its \emph{adjoint} is defined by  
\[
f^\ast(x):=\sum_{i=0}^e (b_i x)^{p^{-i}}
\in \F_q[x^{p^{-\infty}}], 
\]
where $\F_q[x^{p^{-\infty}}]$ denotes the perfection of 
$\F_q[x]$. 

Let $R(x)$ be a nonzero $\F_p$-linearized polynomial. The key observation is that a factorization of the form 
\begin{equation}\label{equation: R+Rast= intro}
    R+R^\ast=F^\ast\circ (aF), 
\end{equation}
where $a\in \F_q^\times$ and $F$ is an $\F_p$-linearized polynomial, yields an explicit formula for the $L$-polynomial of the curve $\overline{C}_R$. As a consequence, it becomes straightforward to determine whether
$\overline{C}_R$ is $\F_q$-maximal or $\F_q$-minimal.

\medskip

We now explain the content in more detail. The construction starts with a pair $\rho=(W,a)$, where $W\subset \F_q$ is an $\F_p$-linear subspace  and $a\in \F_q^\times$. 
Write 
\[\prod_{t \in W} (x-t)=\sum_{i=0}^e c_i x^{p^i}\] and 
define 
\[F_W(x):=\sum_{i=0}^e c_i^{p^{-i}} x^{p^{e-i}}. \]
There exists a unique $\F_p$-linearized polynomial 
$R(x)$ satisfying 
\eqref{equation: R+Rast= intro}. 
We define $D_{\rho}$ to be the curve $C_R$. 

Let $\overline{D}_\rho$ denote the smooth compactification of the curve $D_\rho$. Our first theorem gives an explicit formula for the
$L$-polynomial of $\overline{D}_\rho$. 

Let $\Tr_{q/p} \colon \F_q \to \F_p$ denote the trace map. 
Fix a prime number $\ell$ different from $p_0$. Set $\F_p^\vee:=\Hom_{\mathbb{Z}}(\F_p,\Ql^\times)$. 
For a character $\psi \in \F_p^{\vee}$, set  
 $\psi_q:=\psi \circ \Tr_{q/p}$. 
  For a power $r$ of $p_0$, let $\left(\frac{\cdot}{r}\right)$ denote the unique nontrivial character $\F_r^\times\to\{\pm1\}$.  
  Let $\F$ be an algebraic closure of $\F_q$. 
\begin{theorem}[{Theorem~\ref{lp}}]
Assume that $p_0$ is odd. Let $n\coloneqq[\F_q:\F_p]$. 
For each $\psi\in \F_p^\vee\setminus\{1\}$, set 
\[
G_{\psi}:=-\sum_{x \in \F_p} \psi(x^2). 
\]
Then the $L$-polynomial  
\[
L(\overline{D}_{\rho}/\F_q, T):=\det(1-\mathrm{Fr}_q  T \mid 
H^1(\overline{D}_{\rho,\F},\Ql)) 
\]
satisfies  
\[
L(\overline{D}_{\rho}/\F_q, T)=\prod_{\substack{\psi\in\F_p^\vee\setminus\{1\} \\ v\in W}} (1-\tau_{\psi,v} T), 
\]
where 
\[
\tau_{\psi,v}:=\psi_q(-2^{-1} a^{-1} v^2) \cdot \left(\frac{2a}{q}\right) \cdot 
G_{\psi}^n. 
\]
\end{theorem}

The following theorem 
provides an explicit construction 
of $\F_q$-extremal van der Geer--van der Vlugt curves. Recall that $[\F_q:\F_p]$ is even if there exists an $\F_q$-extremal van der Geer--van der Vlugt curve. 
\begin{theorem}[{Theorem~\ref{t1}}]\label{theorem:  intro}
Assume that $p_0$ is odd and that $[\F_q:\F_p]$ is even. Then the following hold.  
\begin{itemize}
\item[{\rm (i)}] Let $\rho=(W,a)$ be a pair as above. 
Then the curve $\overline{D}_\rho$ is $\F_q$-extremal if and only if 
\[\Tr_{q/p}(a^{-1}v^2)=0 \qquad\text{for all }v\in W.\]
Moreover, if these conditions are satisfied, then $\overline{D}_\rho$ is $\F_q$-maximal (resp.\ $\F_q$-minimal) if and only if 
    \[
    \left(\frac{a}{q}\right)\cdot \left(\frac{-1}{\sqrt{q}}\right)=-1 \qquad\text{(resp. }\left(\frac{a}{q}\right)\cdot \left(\frac{-1}{\sqrt{q}}\right)=1). 
    \]
\item[{\rm (ii)}] 
Let $R(x)\in\F_q[x]$ be an $\F_p$-linearized polynomial. Assume that  $\overline{C}_R$ is 
$\F_q$-extremal. Then there exists a pair $\rho=(W,a)$ satisfying 
\[R=R_\rho. \]
Consequently, we have $C_{R}=D_\rho$. 
\end{itemize}
\end{theorem}

\begin{remark}
In \cite{ITT}, the authors studied van der Geer--van der Vlugt curves in odd characteristic by a different method. They gave a criterion for $\overline{C}_R$ to be $\F_q$-maximal or $\F_q$-minimal in \cite[Theorem~4.11]{ITT}. An important ingredient of their criterion is \cite[Condition~4.6]{ITT}. After translating their notation into ours, one can show that this condition is equivalent to the condition appearing in Theorem~\ref{theorem: intro}(i) (cf.\ \S\ref{ssec: Comparison with ITT}). Thus, part~(i) is essentially equivalent to \cite[Theorem~4.11]{ITT}. 

From this perspective, the main new contribution of Theorem~\ref{theorem: intro} is part~(ii), which shows that every $\F_q$-extremal van der Geer--van der Vlugt curve arises from a pair $\rho=(W,a)$ as above. 
\end{remark}

As a consequence of Theorem~\ref{theorem:  intro}, we obtain a simple construction of $\F_q$-maximal curves.
\begin{theorem}[Corollary \ref{maxc}]
Assume that $p_0$ is odd and that $[\F_q:\F_p]$ is even. Let 
$\sqrt{q}:=p^{[\F_q:\F_p]/2}$. 
Let $\xi \in \F_q$ be such that $\xi^{\sqrt{q}-1}=-1$.
Let 
$\eta \in \F_q^{\times}$ and set 
$a:=\xi^{-1} \eta^2 \in \F_q^{\times}$. Let $W$ be an $\F_p$-vector subspace of $\eta\, \F_{\sqrt{q}}$. 
Then the curve $\overline{D}_{(W,a)}$ is $\F_q$-maximal of genus 
$(\# W) \cdot (p-1)/2$. 
\end{theorem}

 In Section~3, we review the characteristic-two analogues of Theorems~1.1 and~1.2 established in \cite{ITT1}, and highlight the similarities and differences between the odd- and even-characteristic cases.

In Section \ref{appC}, as an application of Theorem~1.2 and its characteristic-two counterpart, 
 we reprove several results from \cite{CO0}. 
More precisely, we show that there exist $\F_q$-maximal and $\F_q$-minimal curves of prescribed genus within the family of van der Geer--van der Vlugt curves (cf.\ Theorem \ref{theorem: existence of extremal curves}). 
We also prove that an $\F_q$-maximal 
van der Geer--van der Vlugt curve  is a finite quotient of 
the Hermitian curve defined by $y^{\sqrt{q}}+y=x^{\sqrt{q}+1}$ (cf.\ Theorem \ref{theorem: quotient of Hermitian curve}).

\section{Preliminaries}
\subsection{$\F_p$-linearized polynomials}
Let $\F_q[x^{p^{-\infty}}]$ denote the perfection of the polynomial ring $\F_q[x]$. Define $\mathcal R$ to be the subset of $\F_q[x^{p^{-\infty}}]$ consisting of all finite sums of the form 
\[
\sum_{i\in \Z}a_ix^{p^i} \qquad(a_i\in\F_q). 
\]
An element of $\mathcal R$ of the form $\sum_{i\geq0} a_i x^{p^i}$ is called an \emph{$\F_p$-linearized polynomial}.

For two elements $f,g\in \mc R$, define  
\[
f\circ g\coloneqq f(g(x)). 
\]
It is straightforward to check that $f\circ g\in \mc R$. Moreover, $\mc R$ becomes a non-commutative ring under the usual addition and the multiplication $\circ$.

For an element $f=\sum_i a_ix^{p^i}\in\mc R$,  define its \emph{adjoint} by
\[
f^\ast(x)\coloneqq\sum_{i} (a_i x)^{p^{-i}}=
\sum_{i} a_i^{p^{-i}} x^{p^{-i}}. 
\]
The operator $(-)^\ast$ satisfies the relations 
\[
(f+g)^\ast=f^\ast+g^\ast,\quad (f\circ g)^\ast=g^\ast\circ f^\ast,\quad f^{\ast\ast}=f,\quad a^\ast=a\quad(f,g\in\mc R,\ a\in\F_q). 
\]

We now recall some basic properties of the adjoint. 
Let $A$ be an $\F_{p}$-algebra. For $f,g\in A$, 
write 
\[
f\sim_A g
\]
if there exists $h\in A$ such that
\[
f-g=h^p - h. 
\]
When $A$ is clear from the context, we simply write $\sim$ instead of $\sim_A$.

\begin{lemma}\label{lem: basic property of sim}
The following statements hold. 
\begin{enumerate}
    \item Let $B$ be the perfection of $A$. For $f,g\in A$, regard $f$ and $g$ also as elements of $B$. Then 
    \[
    f\sim_A g\qquad\text{if and only if}\qquad f\sim_{B}g. 
    \]
\item Let $f(x)\in \mc R$, and let $A$ be the perfection of $\F_q[x,y]$. 
Then 
\[
xf(y)\sim_{A}f^\ast(x)y. 
\]
\end{enumerate}
\end{lemma}
\begin{proof}
    (i)  Let $\wp(x)\coloneqq x^p-x$. 
    By  Artin--Scheier theory, we have 
    \[
    H_\et^1({\rm Spec}(A),\F_p)\cong A/\wp(A),\qquad H_\et^1({\rm Spec}(B),\F_p)\cong B/\wp(B). 
    \]
    The assertion therefore follows from the invariance of \'etale cohomology under perfection, which yields an isomorphism 
    \[
    H_\et^1({\rm Spec}(A),\F_p)\overset{\cong}{\to}H_\et^1({\rm Spec}(B),\F_p). 
    \]

\noindent
     (ii) By linearity, it suffices to treat the case $f(x)=ax^{p^i}$. For every $z\in A$, we have $z^p\sim_A z$. Hence, 
     \[
     xf(y)=xay^{p^i}\sim_A(xa)^{p^{-i}}y=f^\ast(x)y. 
     \]
   The assertion follows. 
\end{proof}

Each element $f\in\mc R$ defines an $\F_p$-linear endomorphism of $\F$. We write 
\[
\ker f\coloneqq \{x\in\F\mid f(x)=0\}. 
\]
\begin{lemma}\label{exact}
Let $f\in\mc R$, and set $V_q\coloneqq\ker f\cap\F_q$. 
Then the sequence 
\[
\F_q\xrightarrow{f^\ast}\F_q\xrightarrow{
a\mapsto \bigl(x\mapsto \Tr_{q/p}(ax)\bigr)} \Hom_{\F_p}(V_q,\F_p)\to0 
\]
is exact. 
\end{lemma}
\begin{proof}
Any $\F_p$-linear map $V_q \to \F_p$ extends to 
an $\F_p$-linear endomorphism of $\F_q$. Moreover, every $\F_p$-linear endomorphism of $\F_q$ is of  the 
form $x \mapsto \Tr_{q/p}(ax)$ for some $a \in \F_q$. Therefore, 
the homomorphism $\F_q\to \Hom_{\F_p}(V_q,\F_p)$ is surjective. 

Let $u \in \F_q$ and $x\in V_q$. Since $f^\ast(u)x\sim uf(x)=0$, 
we have  
\[\Tr_{q/p}(f^\ast(u) x)=0.\]
Hence the sequence 
\[
\F_q\overset{f^\ast}{\to}\F_q\to\Hom_{\F_p}(V_q,\F_p)
\]
is a complex. Let 
\[
\varphi\colon \F_q/f^\ast(\F_q) \to \Hom_{\F_p}(V_q,\F_p)
\]
denote the induced map. 
It remains to show that $\varphi$ is an isomorphism. Since $\varphi$ is
surjective, it suffices to show that the source and target have the same
cardinality. We have 
\[
\#(\ker f^\ast\cap \F_q)=
\#(\F_q/f^\ast(\F_q))\ge \#V_q=\#(\ker f \cap \F_q). 
\]
Applying the same argument with $f^\ast$ in place of $f$, we obtain
 \[\#(\ker f^\ast\cap \F_q)\le \#(\ker f \cap \F_q)=\#V_q. \]
 Therefore, 
\[
\#(\F_q/f^\ast(\F_q))= \#V_q. 
\]
 This completes the proof. 
\end{proof}

We now prove two lemmas that are key ingredients for our main results. 

\begin{lemma}\label{lem: Fast aF=R+Rast}
Assume that $p_0$ is odd. 
Let $a\in \F_q^\times$ and $F\in\mc R$ be an $\F_p$-linearized polynomial. Then the following hold. 
\begin{enumerate}
    \item The composition $F^\ast \circ (aF)$ can be written uniquely in the form
\begin{equation}\label{r}
F^\ast \circ (aF)=R+R^\ast,
\end{equation}
where $R(x)$ is an $\F_p$-linearized polynomial. 
\item Let $R(x)$ be as in {\rm (i)}. Then 
\[
a F(x)^2 \sim_{\F_q[x]} 2 x R(x). 
\]
\end{enumerate}
\end{lemma}
\begin{proof}
(i) Write 
\[
F^\ast\circ(aF)=\sum_ia_ix^{p^i}. 
\]
Since $(F^\ast\circ(aF))^\ast=F^\ast\circ(aF)$, we have $a_{-i}=a_i^{p^{-i}}$ for every $i\geq0$. 
Therefore, 
\[
R(x)=\sum_{i\geq1}a_i x^{p^i}+\frac{a_0}{2}x
\]
is an $\F_p$-linearized polynomial satisfying \eqref{r}. The uniqueness follows immediately from the construction. 

\noindent
(ii) Let $A\coloneqq \F_q[x^{p^{-\infty}}]$. By Lemma~\ref{lem: basic property of sim}(i), it suffices to show that 
\[aF(x)^2\sim_{A}2xR(x). \]
We compute 
\begin{align*}
aF(x)^2=F(x)\bigl(aF(x)\bigr)\sim_AxF^\ast\bigl(aF(x)\bigr)=x(R+R^\ast)(x). 
\end{align*}
Since $xR^\ast(x)\sim_A R(x)x$, we obtain 
\[
aF(x)^2\sim_A2xR(x). 
\]
The assertion follows. 
\end{proof}

We recall the characteristic-two analogue of the above lemma.  For an $\F_2$-algebra $A$, let $W_2(A)$ denote the ring of Witt vectors of length two. As a set, 
\[
W_2(A)=A\times A
\]
with addition and multiplication given by 
\[
(a,b)+(c,d)=(a+c,b+d+ac),\qquad (a,b)\cdot(c,d)=(ac,a^2d+c^2b). 
\]

Assume that $p_0=2$  and that $A$ is an $\F_p$-algebra. 
For $\alpha,\beta\in W_2(A)$, write 
\[
\alpha\sim\beta
\]
if there exists $(a,b)\in W_2(A)$ such that 
\[
\alpha-\beta=(a^p,b^p)-(a,b). 
\]
\begin{lemma}\label{lem: Fast F=R+Rast}
Assume that $p_0=2$. 
Let $F\in\mc R$ be an $\F_p$-linearized polynomial satisfying $F^\ast(1)=0$. Then the following hold. 
\begin{enumerate}
    \item There exists a unique $\F_p$-linearized polynomial $R(x)\in \F_q[x]$ of the form 
    \[
    R(x)=\sum_{i\geq1}a_ix^{p^i}\]
    such that 
\begin{equation*}
F^\ast \circ F=R+R^\ast. 
\end{equation*}
\item 
Let $R(x)$ be as in {\rm (i)}. Write 
\[
F(x)=\sum_{i=0}^eb_ix^{p^i}, 
\]
and define 
\[
\gamma_F\coloneqq\sum_{0\le i<j\le e}b_i^{p^{-i}}b_j^{p^{-j}}. 
\]
Then, in the ring $W_2(\F_q[x])$, we have 
\[
(F(x),0)\sim (0,x R(x)+\gamma_Fx^2). 
\]
\end{enumerate}
\end{lemma}
\begin{proof}
(i) Since $(F^\ast\circ F)^\ast=F^\ast\circ F$, we may write  
\[
F^\ast\circ F=\sum_ia_ix^{p^i}
\]
with $a_{-i}=a_i^{p^{-i}}$. Moreover, writing $F(x)=\sum_{i=0}^eb_ix^{p^i}$, we compute 
\[
a_0=\sum_{i=0}^e(b_i^2)^{p^{-i}}=F^\ast(1)=0. 
\]
    Therefore, $R(x):=\sum_{i\ge1}a_ix^{p^i}$ satisfies the required property. 

    \medskip 

    \noindent
    (ii) The assertion is proved in \cite[Lemma~4.6]{ITT1}. 
\end{proof}

\subsection{Artin--Schreier sheaves}
Fix  a prime number $\ell$ not dividing $p$, and set \[\F_p^{\vee}:=\Hom_{\mathbb{Z}}(\F_p,\Ql^{\times}).\]
For each $\psi \in \F_p^{\vee}$,  let $\mathcal{L}_{\psi}$ denote the 
$\Ql$-sheaf on $\mathbb{A}^1_{\F_q}=\Spec (\F_q[x])$ defined by the \'etale $\F_p$-torsor 
$y^p-y=x$ and $\psi$. 
For a morphism of $\F_q$-schemes  
$f \colon X \to \mathbb{A}^1_{\F_q}$, let $\mathcal{L}_{\psi}(f)$
denote the pullback of $\mathcal{L}_{\psi}$ along $f$. 

We will repeatedly use the following fact.
For an $\F_q$-scheme $X$ and any $f,g\in\Gamma(X,\mc O_X)$ satisfying
\[
f\sim_{\Gamma(X,\mc O_X)}g, 
\]
we have an isomorphism of $\Ql$-sheaves on $X$
\[
\mathcal{L}_{\psi}(f) \cong \mathcal{L}_{\psi}(g). 
\]

\begin{lemma}\label{fpush}
Let $F(x) \in \F_q[x]$ be a separable $\F_p$-linearized polynomial of degree $p^e$ and set 
$W \coloneqq \ker F^\ast$. Assume that $W \subset \F_q$. 
Regard $F$ as an $\F_q$-morphism $\A^1_{\F_q}\to \A^1_{\F_q},\ x \mapsto F(x)$. 
Then there is an isomorphism 
\[
\bigoplus_{v \in W} 
\mathcal{L}_{\psi}(vx) \xrightarrow{\cong}
F_{\ast} \Ql. 
\]
\end{lemma}
\begin{proof}
For every $v \in W=\ker F^\ast$, we have 
$vF(x) \sim_{\F_q[x]} F^\ast(v)x=0$. 
Hence, 
\[
F^{-1}\mathcal{L}_{\psi}(vx)
\cong \mathcal{L}_{\psi}(vF(x)) 
\cong \mathcal{L}_{\psi}(F^\ast(v)x) 
\cong \Ql. 
\]
By adjunction, this yields a nonzero morphism 
\[
\mathcal{L}_{\psi}(vx) \to F_{\ast}\Ql.
\]
Taking the direct sum over all $v \in W$, we obtain a morphism
\[
\bigoplus_{v \in W} \mathcal{L}_{\psi}(vx) \to F_{\ast} \Ql.
\]
We claim that this morphism is an isomorphism. 
To prove injectivity, we note that the summands 
$\mathcal{L}_{\psi}(vx)$ are pairwise non-isomorphic rank-one smooth 
$\Ql$-sheaves for distinct $v \in W$. Since every morphism $\mathcal{L}_{\psi}(vx) \to F_{\ast} \Ql$ is nonzero, the induced morphism is injective.

Finally, both sides are smooth $\Ql$-sheaves of rank $\#W$. Therefore,  the above injective morphism is necessarily an isomorphism.
\end{proof}

\section{Construction of van der Geer--van der Vlugt curves}
Let $p$ be a power of $p_0$, and $q$ be a power of $p$. Let $R(x)\in \F_q[x]$ be a nonzero $\F_p$-linearized polynomial. The van der Geer--van der Vlugt curve associated with $R(x)$ is the smooth affine curve over $\F_q$ defined by 
\[
C_R\colon y^p-y=xR(x). 
\]

We recall some basic geometric properties of the curves $C_{R}$. 
\begin{lemma}\label{lem: geometric properties of CRs}
   The curve $C_{R}$  is smooth and geometrically connected. Its smooth  compactification $\overline{C}_{R}$ has genus $(\deg R) \cdot (p-1)/2$. 
Moreover, the complement $\overline{C}_{R}\setminus C_{R}$ consists of a single $\F_q$-rational point.  
\end{lemma}
\begin{proof}
    Smoothness follows from the Jacobian criterion. The other properties follow readily from \cite[Proposition 6.4.1]{St}. 
\end{proof}

We also recall the following result. 
\begin{lemma}\label{lemma: When CR is extremal Fq Fp is even and VR in Fq}
Assume that $\overline{C}_R$ is $\F_q$-extremal. Then $[\F_q:\F_p]$ is even, and  
\begin{equation}\label{vr}
V_R:=\{x\in \F\mid R(x)+R^\ast(x)=0\}
\end{equation}
 is contained in $\F_q$.    
\end{lemma}
\begin{proof}
The assertion follows from \cite[Lemma~A.3]{ITT1}. Indeed, it is proved there that
$[\F_q:\F_p]$ is even and that
\[
H_R :=\{(a,b)\in V_R\times \F\mid b^p-b=aR(a)\}
\]
is contained in $\F_q\times \F_q$.
Since the projection
\[
H_R\to V_R,\qquad (a,b)\mapsto a,
\]
is surjective, it follows that $V_R\subset \F_q$.
\end{proof}

In this section, we give a recipe for constructing van der Geer--van der Vlugt curves that is suitable for computing their $L$-polynomials. The construction is divided according to the parity of $p_0$.

\subsection{The case where $p_0$ is odd}\label{ssec: The case where p0 is odd}
In this subsection, we assume that $p_0$ is odd. In this case, the construction starts from a pair 
\[\rho=(W,a),\]
where $W\subset \F_q$ is an $\F_p$-linear subspace and $a\in \F_q^\times$. 

Let 
\[
f_W(x)\coloneqq\prod_{v\in W}(x-v).  
\]
Define 
\begin{equation}\label{fwd}
F_W(x)\coloneqq f_W^\ast(x^{p^e}), 
\end{equation}
where $\dim_{\F_p}W=e$. 
It is straightforward to check that $F_W$ is a separable $\F_p$-linearized polynomial. 
Then
\[
F_W^\ast(x)=f_W(x)^{p^{-e}}.
\]
In particular,
\[
W=\ker F_W^\ast.
\]
By Lemma~\ref{lem: Fast aF=R+Rast}, there exists a unique $\F_p$-linearized polynomial $R_\rho(x) \in \F_q[x]$ 
such that 
\[R_\rho+R_\rho^\ast=F_W^\ast \circ (a F_W).\]
Let
\[
D_{\rho}\colon y^p-y=xR_\rho(x)
\]
be the curve associated with $R_\rho$. 
  
In the following theorem, we give a formula for the 
$L$-polynomial of the smooth compactification $\overline{D}_{\rho}$ of $D_\rho$. 
For a character $\psi \in \F_p^{\vee}$, let 
 $\psi_q:=\psi \circ \Tr_{q/p}$. 
 For a $p_0$-power $r$, let 
 \[
 \left(\frac{\cdot}{r}\right)
 \colon \F_r^{\times} \to \{\pm 1\},\quad x \mapsto x^{\frac{r-1}{2}}
 \]
 denote the Legendre symbol. 
\begin{theorem}\label{lp}
Set $n\coloneqq[\F_q:\F_p]$. 
For each $\psi\in \F_p^\vee\setminus\{1\}$, set 
\[
G_{\psi}:=-\sum_{x \in \F_p} \psi(x^2). 
\]
Then the $L$-polynomial  
\[
L(\overline{D}_{\rho}/\F_q, T):=\det(1-\mathrm{Fr}_q  T \mid 
H^1(\overline{D}_{\rho,\F},\Ql)) 
\]
satisfies  
\[
L(\overline{D}_{\rho}/\F_q, T)=\prod_{\substack{\psi\in\F_p^\vee\setminus\{1\} \\ v\in W}} (1-\tau_{\psi,v} T), 
\]
where 
\[
\tau_{\psi,v}:=\psi_q(-2^{-1} a^{-1} v^2) \cdot \left(\frac{2a}{q}\right) \cdot 
G_{\psi}^n. 
\]
\end{theorem}
\begin{proof}
By Lemma~\ref{lem: geometric properties of CRs}, the canonical morphism 
\[
H^1_{\mathrm{c}}(D_{\rho,\F},\Ql)\to H^1(\overline{D}_{\rho,\F},\Ql)
\]
is an isomorphism. Hence, it suffices to compute $\det(1-{\rm Fr}_q T\mid H^1_{\mathrm{c}}(D_{\rho,\F},\Ql))$.

Regard $F_W$ as an $\F_q$-morphism $\A^1_{\F_q}\to \A^1_{\F_q}$. 
By Lemmas \ref{lem: basic property of sim}(ii) and \ref{lem: Fast aF=R+Rast}(ii), we compute 
\begin{equation}\label{fr}
F_W^{-1}\mathcal{L}_\psi(2^{-1}a x^2)\cong 
\mathcal{L}_\psi(2^{-1} a F_W(x)^2)\cong 
\mathcal{L}_\psi(xR_\rho(x)). 
\end{equation}
Consider the $\F_q$-morphism 
\[\pi\colon D_{\rho}\to \A^1_{\F_q},\ (x,y)\mapsto x,\]
which is 
a finite \'etale Galois covering with Galois group $\F_p$. We have \[
\pi_{\ast} \Ql \cong 
\bigoplus_{\psi \in \F_p^{\vee}} \mathcal{L}_{\psi}(x R_\rho(x)). \]
Hence, we obtain 
\begin{align*}
H^1_{\mathrm{c}}(D_{\rho,\F},\Ql)
&\cong
H_{\rm c}^1(\mathbb{A}^1,\pi_{\ast} \Ql)\\
&\cong  \bigoplus_{\psi\in\F_p^\vee\setminus\{1\}}H^1_{\mathrm{c}}(\A^1,\mathcal{L}_\psi(xR_\rho(x)))\\
&\cong  \bigoplus_{\psi\in\F_p^\vee\setminus\{1\}}H^1_{\mathrm{c}}(\A^1,F_{W\ast} F_W^{-1} \mathcal{L}_\psi(2^{-1}ax^2)) \\
&\cong \bigoplus_{\psi\in\F_p^\vee\setminus\{1\}}H^1_{\mathrm{c}}(\A^1,\mathcal{L}_\psi(2^{-1} a x^2)\otimes F_{W\ast}\Ql)\\
& \cong \bigoplus_{\psi\in\F_p^\vee\setminus\{1\}}\bigoplus_{v\in W}H^1_{\mathrm{c}}(\A^1,\mathcal{L}_\psi(2^{-1} a x^2+vx)), 
\end{align*}
where the second isomorphism follows from 
 $H_{\rm c}^1(\mathbb{A}^1,\Ql)=0$, 
the third one follows from \eqref{fr},  
the fourth one follows from
the projection formula, and 
the last one follows from Lemma \ref{fpush} for $F=F_W$.

By the Grothendieck--Lefschetz trace formula, for every nontrivial $\psi$, we have 
\[
 \Tr(\mathrm{Fr}_q \mid H^1_{\mathrm{c}}(\A^1,\mathcal{L}_\psi(2^{-1}ax^2+vx)))=-\sum_{x\in\F_q}\psi_q(2^{-1}ax^2+vx). 
\]
Completing the square, we obtain 
\begin{align*}
 -\sum_{x\in\F_q}\psi_q(2^{-1}ax^2+vx)&=-\sum_{x\in\F_q}\psi_q(2^{-1}a(x+a^{-1}v)^2-2^{-1}a^{-1}v^2)\\
&=\psi_q(-2^{-1}a^{-1}v^2) \cdot \left(\frac{2^{-1}a}{q}\right) \cdot \biggl(\,-\sum_{x \in \F_q} \psi_q(x^2)\biggr) \\
&=\psi_q(-2^{-1}a^{-1}v^2) \cdot \left(\frac{2a}{q}\right)G_{\psi}^n=\tau_{\psi,v},  
\end{align*}
where the third equality follows from 
the Hasse--Davenport theorem. 
Since the cohomology group
\[
H^1_{\mathrm{c}}(\A^1,\mathcal{L}_\psi(2^{-1}ax^2+vx))\] has dimension $1$, we obtain 
\[
\det(1-{\rm Fr}_q T\mid H^1_{\mathrm{c}}(D_{\rho,\F},\Ql))=\prod_{\substack{\psi\in\F_p^\vee\setminus\{1\} \\ v\in W}} (1-\tau_{\psi,v} T). 
\]
This completes the proof.  
\end{proof}

The exponential sum associated with $x R_\rho(x)$ 
can be expressed via partial quadratic sums. 
\begin{corollary}\label{expsum}
Let $n=[\F_q:\F_p]$. 
For $\psi \in \F_p^{\vee} \setminus \{1\}$, we have 
\[
\sum_{x \in \F_q} \psi_q(x R_\rho(x))=-\left(\frac{2a}{q}\right) 
G_{\psi}^n
\sum_{v\in W}\psi_q(-2^{-1}a^{-1}v^2). 
\]
\end{corollary}
\begin{proof}
The claim follows from Theorem \ref{lp}. 
\end{proof}

\begin{theorem}\label{t1}
Assume that $[\F_q:\F_p]$ is even. Then the following hold.  
\begin{itemize}
\item[{\rm (i)}]  Let $\rho=(W,a)$ be a pair as above. 
Then the curve $\overline{D}_\rho$ is $\F_q$-extremal if and only if 
\[\Tr_{q/p}(a^{-1}v^2)=0 \qquad\text{for all }v\in W.\]
Moreover, if these conditions are satisfied, then $\overline{D}_\rho$ is $\F_q$-maximal (resp. $\F_q$-minimal) if and only if 
    \[
    \left(\frac{a}{q}\right)\cdot \left(\frac{-1}{\sqrt{q}}\right)=-1 \qquad\text{(resp. }\left(\frac{a}{q}\right)\cdot \left(\frac{-1}{\sqrt{q}}\right)=1). 
    \]
\item[{\rm (ii)}] 

Let $R(x)\in\F_q[x]$ be an $\F_p$-linearized polynomial. Assume that  $\overline{C}_R$ is 
$\F_q$-extremal. Then there exists a pair $\rho=(W,a)$ satisfying 
\[R=R_\rho. \]
Consequently, we have $C_{R}=D_\rho$. 
\end{itemize}
\end{theorem}
\begin{proof}
(i) Since $n=[\F_q:\F_p]$ is even, we have $\displaystyle \left(\frac{2a}{q}\right) = \left(\frac{a}{q}\right)$ and 
  \[
  G_{\psi}^n=\left\{\left(\frac{-1}{p}\right)p \right\}^{n/2}=
  \left(\frac{-1}{\sqrt{q}}\right)\sqrt{q}. 
  \]
  Hence we obtain 
  \[
\tau_{\psi,v}=\psi_q(-2^{-1}a^{-1}v^2) \cdot \left(\frac{a}{q}\right) \cdot 
\left(\frac{-1}{\sqrt{q}}\right)\sqrt{q}.
  \]
  Therefore, the curve $\overline{D}_{\rho}$ is $\F_q$-extremal if and only if 
  \[
  \psi(\Tr_{q/p}(-2^{-1}a^{-1}v^2))=
  \psi_q(-2^{-1}a^{-1}v^2)=1 \quad \textrm{for all $\psi \in \F_p^{\vee} \setminus \{1\}$ and 
  all $v \in W$}. 
  \]
  This condition is equivalent to 
  \begin{equation*}\label{equation: condition for ex}
      \Tr_{q/p}(a^{-1}v^2)=0 \quad \textrm{for all $v \in W$}. 
  \end{equation*}

Assume now that these equivalent conditions are satisfied. Then 
\[
\tau_{\psi,v}=\left(\frac{a}{q}\right) \cdot 
\left(\frac{-1}{\sqrt{q}}\right)\sqrt{q}.
  \]
  This completes the proof.

  \medskip 

  \noindent
  (ii) Let $R(x)\in\F_q[x]$ be an $\F_p$-linearized polynomial of degree $p^e$.  Define 
  \[E:=R+R^\ast\qquad\text{and}\qquad V:=
  \ker E=\{x \in\F \mid E(x)=0\}.\]
  Assume that $\overline{C}_{R}$ is $\F_q$-extremal. Then, 
  by Lemma~\ref{lemma: When CR is extremal Fq Fp is even and VR in Fq}, we have 
 $V\subset \F_q$. By \cite[Lemma~2.6]{ITT1}, $V$ carries a canonical nondegenerate symplectic pairing. 
  Then, by choosing a maximal totally isotropic subspace of $V$ 
  and applying \cite[Corollary 2.10]{ITT1}, one can write 
  \[
  E=F^\ast \circ (aF)
  \]
  for some $a\in \F_q^\times$ and $F(x)=\sum_{i=0}^eb_ix^{p^i}\in \F_q[x]$ with $b_e\neq0$ and $b_0=1$.

  Let $W:=\ker F^\ast$. 
  Since $aF$ defines a surjection 
  $V \to W$ and $V\subset \F_q$, it follows that $W \subset \F_q$. 

Set $\rho:=(W,a)$. 
Since $b_e\neq0$ and $b_0=1$, the polynomial $F^\ast(x)^{p^e}$ is monic and separable. 
Hence 
\[
f_W(x)=\prod_{v\in W}(x-v)=F^\ast(x)^{p^e}. 
\]
Taking the adjoint gives $f_W^\ast(x)=F(x^{p^{-e}})$. 
Substituting $x^{p^e}$ for $x$ yields 
\[
F_W(x)=f_W^\ast(x^{p^e})=F(x).
\]
Hence 
\[
R_{\rho}+R_{\rho}^\ast=F_W^{\ast}\circ (a F_W)=F^\ast \circ (aF)=
R+R^\ast. 
\]
  This shows that $R=R_\rho$ by the uniqueness in Lemma \ref{lem: Fast aF=R+Rast}(i). The assertion follows. 
 \end{proof} 
The following gives a simple recipe for constructing $\F_q$-maximal curves. See \cite[Corollary 6.7]{ITT1} for the characteristic-two analogue.
\begin{corollary}\label{maxc}
Assume that $[\F_q:\F_p]$ is even. 
Let $\xi \in \F_q^{\times}$ be such that $\xi^{\sqrt{q}-1}=-1$.
Let $\eta \in \F_q^{\times}$ and set 
$a:=\xi^{-1} \eta^2$. 
Let $W$ be an $\F_p$-vector subspace of $\eta\, \F_{\sqrt{q}}$ and let $\rho \coloneqq (W,a)$. 
Then $\overline{D}_{\rho}$ is $\F_q$-maximal. 
\end{corollary}

\begin{proof}
Since $a=\xi^{-1} \eta^2$ and $\xi^{\sqrt{q}-1}=-1$, we have
\[
\left(\frac{a}{q}\right)=\left(\frac{\xi}{q}\right)
=\xi^{(q-1)/2}
=(\xi^{\sqrt{q}-1})^{(\sqrt{q}+1)/2}
=(-1)^{(\sqrt{q}+1)/2}
=-\left(\frac{-1}{\sqrt{q}}\right).
\]

Note that $\Tr_{q/\sqrt{q}}(\xi)=\xi^{\sqrt{q}}+\xi=0$. 
For $v=\eta x \in \eta\, \F_{\sqrt{q}}$, using $a=\xi^{-1} \eta^2$, we obtain
\[
\Tr_{q/\sqrt{q}}(a^{-1}v^2)
=\Tr_{q/\sqrt{q}}(\xi x^2)
=x^2 \Tr_{q/\sqrt{q}}(\xi)
=0.
\]
Hence $\Tr_{q/p}(v^2/a)=0$ for 
$v \in \eta\, \F_{\sqrt{q}}$. 
In particular, since $W \subset \eta\, \F_{\sqrt{q}}$, we have
\[
\Tr_{q/p}(a^{-1}v^2)=0 
\quad \text{for all } v \in W.
\]
Therefore the claim follows from Theorem~\ref{t1}(i).
\end{proof}

\subsection{The case where $p_0=2$}
The construction in this case is given in \cite{ITT1}. Let us review the result. 

We consider a pair 
\[
\epsilon=(W,t)
\]
where $W\subset \F_q$ is an $\F_p$-linear subspace containing $1$, and $t\in\F_q$. 

Define $f_W(x)$ and $F_W(x)$ as in \S\ref{ssec: The case where p0 is odd}: 
\[
f_W(x)\coloneqq\prod_{v\in W}(x-v),\qquad 
F_W(x)\coloneqq f_W^\ast(x^{p^e}), 
\]
where $\dim_{\F_p}W=e$. Write 
\[
F_W(x)=\sum_{i=0}^eb_ix^{p^i}, 
\]
and define 
\[
\gamma_{F_W}\coloneqq\sum_{0\le i<j\le e}b_i^{p^{-i}}b_j^{p^{-j}}. 
\]

Write 
\[
F_W^\ast\circ F_W=\sum_{i=-e}^ea_ix^{p^i}\qquad(a_i\in\F_q). 
\]
Since $(F_W^\ast\circ F_W)^\ast=F_W^\ast\circ F_W$, we have $a_{-i}=a_i^{p^{-i}}$. Moreover, by \cite[Corollary~2.10]{ITT1}, we have $a_0=F_W^\ast(1)$, which vanishes since $1\in W$. 

Therefore, the polynomial 
$R(x)=\sum_{i=1}^ea_ix^{p^i}$ satisfies 
\[
R+R^\ast=F_W^\ast\circ F_W. 
\]
Define 
\[
R_\epsilon(x)\coloneqq R(x)+\bigl(\gamma_{F_W}+F_W^\ast(t)^2\bigr)x. 
\]
This also satisfies $R_\epsilon+R_\epsilon^\ast=F_W^\ast\circ F_W$. 

Let 
\[
D_\epsilon: y^p+y=xR_{\epsilon}(x) 
\]
be the associated curve. 
The following theorem is a characteristic-two analogue of Theorem~\ref{lp}. 

We fix a faithful character $\xi_2 \colon W_2(\F_2) \to \Ql^{\times}$ and set $\sqrt{-1}\coloneqq \xi_2(1,0)$. 
We define a composite 
\[
\xi_q \colon W_2(\F_q) \xrightarrow{\Tr_{q/2}} 
W_2(\F_2) \xrightarrow{\xi_2} \Ql^{\times}, 
\]
where $\Tr_{q/2}$ is the trace map defined by 
\[
(x,y) \mapsto \sum_{i=0}^{[\F_q:\F_2]-1}(x^{2^i}, y^{2^i}).
\]
We set 
\[
Q_q(x):=\xi_q(x,0), \qquad \psi_q(x):=\xi_q(0,x) \qquad \textrm{for $x \in \F_q$}. 
\]
\begin{theorem}[{\cite[Theorem~4.12]{ITT1}}]\label{twist}
Let $\overline{D}_\epsilon$ denote the smooth compactification of $D_\epsilon$. Then the $L$-polynomial 
\[
L(\overline{D}_{\epsilon}/\F_q, T):=\det(1-\mathrm{Fr}_q T \mid 
H^1(\overline{D}_{\epsilon,\F},\Ql)) 
\]
satisfies 
\[
L(\overline{D}_{\epsilon}/\F_q, T)=
 \prod_{v \in W}
(1-\tau_{v} T)^{p-1},
\]
where 
\[
\tau_{v}:=Q_q(t+v)^{-1} (-1-\sqrt{-1})^{[\F_q:\F_2]}. 
\]
\end{theorem}

The following theorem is a characteristic-two analogue of Theorem~\ref{t1}. 
\begin{theorem}\label{recipe for ex curve intro}
Assume that $[\F_q:\F_p]$ is even. Then the following statements hold: 
\begin{enumerate}
    \item The curve $\overline{D}_{\epsilon}$ is $\F_q$-extremal if and only if 
    \[
    Q_q(v)=\psi_q(tv)\qquad\text{for all } v\in W. 
    \]
    Moreover, if these conditions are satisfied, then 
    $\overline{D}_{\epsilon}$ is $\F_q$-maximal (resp.\ $\F_q$-minimal) if and only if 
    \[
    Q_q(t)=-(\sqrt{-1})^{[\F_q:\F_2]/2}\qquad\text{(resp.\ }Q_q(t)=(\sqrt{-1})^{[\F_q:\F_2]/2}\text{)}. 
    \]
    \item Let 
\[R(x)=\sum_{i=0}^e a_i x^{p^i} \in 
 \F_q[x],\]
 where $e\geq1$ and $a_e\neq0$, and assume that $\overline{C}_R$ is $\F_q$-extremal. Then there exists a pair $\epsilon=(W,t)$ as above and $a\in\F_q^\times$ such that $(x,y)\mapsto (ax,y)$ induces an isomorphism 
 \[ C_R\xrightarrow{\cong}D_{\epsilon}.
 \]
\end{enumerate}
\end{theorem}
\begin{proof}
By Theorem~\ref{twist}, the curve $\overline{D}_\epsilon$ is $\F_q$-extremal if and only if 
\[
Q_q(v+t)=Q_q(t)\qquad\text{for all }v\in W. 
\]
Using the identity $Q_q(x+y)Q_q(x)^{-1}Q_q(y)^{-1}=\psi_q(xy)$, we can rewrite this condition as 
\[
Q_q(v)=\psi_q(tv)\qquad\text{for all }v\in W.
\]
The remaining assertions follow from \cite[Theorem~1.3]{ITT1}. 
\end{proof}

\begin{remark}
We explain the analogy between Theorems~\ref{t1} and~\ref{recipe for ex curve intro}. 

In Theorem~\ref{recipe for ex curve intro}, we consider the condition 
\[
    Q_q(v)=\psi_q(tv)\qquad\text{for all } v\in W. 
    \]
In particular, this implies that $Q_q|_W$ is a group homomorphism. Hence, 
\begin{equation}\label{even case}
\psi_q(vw)=Q_q(v+w)Q_q(v)^{-1}Q_q(w)^{-1}=1 \qquad\text{for all }v,w\in W.
\end{equation}
 
In the odd-characteristic case, we consider the condition 
\[
\Tr_{q/p}(a^{-1}v^2)=0\qquad\text{for all }v\in W, 
\]
which is equivalent to 
\begin{equation}\label{odd case}
    \Tr_{q/p}(a^{-1}vw)=0\qquad\text{for all }v,w\in W. 
\end{equation}
The conditions \eqref{even case} and \eqref{odd case} are analogous. 

However, in the even-characteristic case, condition~\eqref{even case} does not imply that $Q_q(v)=1$ for all $v\in W$. Therefore, an additional parameter $t$ is needed to encode the group homomorphism $Q_q|_W$.
\end{remark}
\subsection{Comparison with \cite[Condition 4.6]{ITT}}\label{ssec: Comparison with ITT}
We compare \cite[Condition 4.6]{ITT} with the condition in Theorem \ref{t1}(i). 
We use the notation of \cite{ITT} and \cite{TT}. 
\begin{proposition}\label{4.6pre}
Assume that $p$ is odd. 
Let $R(x) \in \F_q[x]$ be an $\F_p$-linearized polynomial of degree $p^e$. Let $V_R$ be as in \eqref{vr}. 
Assume that there exists a maximal totally isotropic subspace $\overline{A}$ in $V_R$ with respect to the symplectic pairing $\omega_R$ in \cite[Lemma 2.3(3)]{TT} such that $\overline{A} \subset \F_q$ (cf.\ \cite[(2.2) and Lemma 2.6]{TT}). 

Let $F_A(x) \in \F_q[x]$ be the $\F_p$-linearized separable polynomial such that 
\begin{itemize}
\item $\ker F_A = \overline{A}$, and 
\item $F_A^\ast(x)^{p^e}$ is monic.  
\end{itemize}
Let $a(x) \in \F_q[x]$ be the $\F_p$-linearized polynomial such that $a \circ F_A = x^q - x$. 
Let $c_A \in \F_q^{\times}$ be a constant defined by 
$R(x)$ and $F_A(x)$ as in \cite[(3.4)]{ITT}. 

Then we have 
\begin{itemize}
\item[{\rm (i)}] $F_A^\ast \circ (2c_A F_A)=R+R^\ast$, 
\item[{\rm (ii)}] $(\ker a)^{\perp} := \left\{x \in \F_q \mid \Tr_{q/p}(xt)=0 \ \text{for all } t \in \ker a\right\}$ is equal to $\ker F_A^\ast$. 
\item[{\rm (iii)}] Let $W\coloneqq \ker F_A^{\ast} \subset \F_q$. Then 
$F_W=F_A$. 
\end{itemize}
\end{proposition}
\begin{proof}
(i) 
By \cite[(A.8)]{TT}, we have 
\[
c_A F_A(x)^2 \sim x R(x). 
\]
We use a similar argument to the proof of Lemma \ref{lem: Fast aF=R+Rast}(ii). 
Since $x R(x) \sim x R^\ast(x)$, 
\[
2 c_A F_A(x)^2 \sim 2 xR(x) \sim x(R(x)+R^\ast(x)). 
\]
Since $2 c_A F_A(x)^2=F_A(x) \cdot 2c_A F_A(x) \sim x (F_A^\ast \circ (2 c_A F_A))$, we obtain 
\[
x (F_A^\ast \circ (2 c_A F_A)-(R+R^\ast)) \sim 0. 
\]
Hence there exists $d \in \F_q[x^{p^{-\infty}}]$ such that 
\[
d^p-d=x (F_A^\ast \circ (2 c_A F_A)-(R+R^\ast)). 
\]
Since the right-hand side has the form 
\[
\sum_{i \in \mathbb{Z}} c_i x^{p^i+1},   
\]
we must have $d^p-d=0$ and hence 
$F_A^\ast \circ (2 c_A F_A)=R+R^\ast$.

(ii) 
By Lemma \ref{exact} for $f(x)=a(x)$, we obtain 
\[
(\ker a)^{\perp}=a^\ast(\F_q). 
\]
Taking the adjoint of $a \circ F_A=x^q-x$ gives 
$F_A^\ast \circ a^{\ast}=x^{q^{-1}}-x$.  
Hence 
\[
a^\ast(\F_q)=\ker F_A^{\ast}. 
\]
Hence, we conclude that $(\ker a)^{\perp}=\ker F_A^\ast$. 

(iii) We have $\deg F_A=p^e$. 
The $\F_p$-linearized polynomial 
$(F_A^\ast)^{p^e}$ is separable and monic. 
Since $F_A^\ast(x)^{p^e}=0$ for $x \in W$, 
\[
f_W(x)=\prod_{v \in W}(x-v)=F_A^\ast(x)^{p^e}. 
\]
Similarly to the proof of Theorem \ref{t1}(ii), we obtain $F_W(x)=f_W^\ast(x^{p^e})=F_A(x)$.  
\end{proof}
\cite[Condition 4.6]{ITT} states that 
\begin{equation}\label{4.6}
\Tr_{q/p}(c_A^{-1} v^2)=0 \quad \textrm{for all $v \in (\ker a)^{\perp}$}. 
\end{equation}
\begin{corollary}\label{c4.6}
Let the notation and the assumptions be as in Proposition \ref{4.6pre}. 
Set $\rho=(W,a) \coloneqq (\ker(F_A^\ast), 2c_A)$.
The condition \eqref{4.6} is equivalent to 
the condition in Theorem {\rm \ref{t1}(i)}:  
\[
\Tr_{q/p}(a^{-1} v^2)=0 \quad \textrm{for all $v \in W$}. 
\]
\end{corollary}
\begin{proof}
From Proposition \ref{4.6pre}, it follows that 
\[
F_W^\ast \circ (a F_W)=R+R^{\ast}, \qquad 
(\ker a)^{\perp}=W.
\]
Hence the assertion follows from 
$\Tr_{q/p}(a^{-1} v^2)=2^{-1}\Tr_{q/p}(c_A^{-1} v^2)$ for $v \in W$. 
\end{proof}

\section{Quotients of van der Geer--van der Vlugt curves}\label{appC}

As an application of the results of the previous section, we give alternative proofs of several results from \cite{CO0}. 
Let $p$ be a power of a prime number $p_0$, and $q$ be a power of $p$. 

All the results in this section will be deduced from the following theorem. 

\begin{theorem}\label{exist12}
Assume that $n=[\F_q:\F_p]$ is even. 
Let $R(x) \in \F_q[x]$ 
be an $\F_p$-linearized polynomial of degree $p^e$. Then the following statements hold. 
\begin{itemize}
\item[{\rm (i)}] 
Assume that the curve $\overline{C}_R$ is $\F_q$-maximal. Assume that $e$ satisfies 
\[
\begin{cases}
0 \le e < n/2 & \textrm{if $p$ is odd}, \\
1 \le e < n/2 &\textrm{if $p$ is even}. 
\end{cases}
\]
Then there exists an $\F_p$-linearized polynomial $R_1(x) \in \F_q[x]$ of degree $p^{e+1}$ with the following properties: 
\begin{itemize}
    \item The curve $\overline{C}_{R_1}$ is $\F_q$-maximal. 
    \item There exists a finite \'etale morphism $C_{R_1}\to C_R$. 
\end{itemize}
\item[{\rm (ii)}] 
Assume that the curve $\overline{C}_R$ is $\F_q$-minimal. Assume that $e$ satisfies 
\[
\begin{cases}
0 \le e < (n/2)-1 & \textrm{if $p$ is odd}, \\
1 \le e < (n/2)-1 &\textrm{if $p$ is even}. 
\end{cases}
\]
Then there exists an $\F_p$-linearized polynomial $R_1(x)\in \F_q[x]$ of degree $p^{e+1}$ with the following properties: 
\begin{itemize}
    \item The curve $\overline{C}_{R_1}$ is $\F_q$-minimal. 
    \item There exists a finite \'etale morphism $C_{R_1}\to C_R$. 
\end{itemize}
\end{itemize}
\end{theorem}
The proof is given in the following subsection. 
\subsection{Proof of Theorem~\ref{exist12}}

Define 
\[
E:=R+R^\ast,\qquad V_R:=\ker E. 
\]
Since $\overline{C}_R$ is $\F_q$-extremal, we have $V_R\subset \F_q$. 
\begin{lemma}\label{lemma: existence of u}
    Let 
    \[
A:=\{u\in \F_q\mid \Tr_{q/p}(uR(u))=0\}. 
\]
Assume that $e<n/2$ if $\overline{C}_R$ is $\F_q$-maximal, and that $e<(n/2)-1$ if $\overline{C}_R$ is $\F_q$-minimal. 

Then the set $A\setminus V_R$ is nonempty. 
\end{lemma}
\begin{proof}
By Artin--Schreier theory, for a given $a\in\F_q$, the equation 
\[
y^p-y=a
\]
has a solution in $\F_q$ if and only if $\Tr_{q/p}(a)=0$. Moreover, when these equivalent conditions are satisfied, there are exactly $p$ solutions. Therefore, 
\[
\# C_R(\F_q)=p\cdot\#A. 
\]

Assume first that $\overline{C}_R$ is $\F_q$-maximal. 
 Since $\overline{C}_R$ has genus $p^e(p-1)/2$ and  $\overline{C}_R\setminus C_R$ consists of a single $\F_q$-rational point by Lemma \ref{lem: geometric properties of CRs}, we have 
\[
\# C_R(\F_q)= q+(p-1)p^e\sqrt{q}. 
\]
Hence
\[
\#A=\frac{q}{p}+(p-1)p^{e-1}\sqrt{q}.
\]
Since $n>2e$, we have
\[
\#A>\frac{q}{p}\ge p^{2e}=\#V_R.
\]
Therefore, $A\setminus V_R$ is nonempty.

Assume now that $\overline{C}_R$ is $\F_q$-minimal. In this case, 
\[
\# C_R(\F_q)= q-(p-1)p^e\sqrt{q}, 
\]
and hence 
\[
\# A=p^{-1}\sqrt{q}\cdot(\sqrt{q}-(p-1)p^e). 
\]
Since $n>2e+2$, we have $\sqrt{q}-(p-1)p^e>0$. Thus, 
\[
\# A>p^{-1}p^{e+1}\cdot(p^{e+1}-(p-1)p^e)=p^{2e}. 
\]
Thus $\#A>\#V_R$, and consequently $A\setminus V_R$ is nonempty. 
\end{proof}

The following lemma is useful to construct a quotient morphism between van der Geer--van der Vlugt curves. 
\begin{lemma}\label{abc}
Let $h,\, R_1,\, R_2 \in \F_q[x]$ be  
$\F_p$-linearized polynomials 
such that 
\[x R_1(x) \sim h(x) R_2(h(x)). \]
Then there exists a finite  \'etale  surjective morphism $C_{R_1} \to C_{R_2}$ defined over $\F_q$. 
 \end{lemma}
 \begin{proof}
 By definition, there exists $g(x) \in \F_q[x]$ such 
that 
\[
h(x) R_2(h(x))=x R_1(x)+g(x)^p-g(x). 
\]
Hence we obtain a morphism 
\[
C_{R_1} \to C_{R_2},\ (x,y) \mapsto (h(x),\, y+g(x)). 
\]
It is straightforward to check that this morphism is finite, \'etale, and surjective. 
\end{proof}

\subsubsection{The case where $p$ is odd}
In this case, we may assume that $C_R=D_{\rho}$ by Theorem \ref{t1}(ii), where $\rho$ is a pair $(W,a)$ of an $\F_p$-subspace $W\subset\F_q$ and an element 
$a\in \F_q^\times$  such that 
\[
\Tr_{q/p}(a^{-1}v^2)=0\qquad\text{for all $v\in W$.}
\]
Note that 
\begin{equation}\label{rr}
R+R^\ast=F_W^\ast \circ (aF_W).
\end{equation}

By Lemma~\ref{lemma: existence of u}, we may choose an element $u\in A\setminus V_R$. 
Set 
\[
W'\coloneqq W+\langle aF_W(u)\rangle_{\F_p}.
\]
Since $u\notin V_R=\ker (F_W^\ast \circ (a F_W))$ and $W=\ker F_W^\ast$, we have $a F_W(u) \not\in W$. 
Hence $\dim_{\F_p} W'=e+1$. 

Set $\rho':=(W',a)$. We show that $\overline{D}_{\rho'}$ is $\F_q$-extremal by verifying the condition in
Theorem~\ref{t1}(i). 
By \eqref{rr},  
\begin{align*}
\Tr_{q/p}(a^{-1} \cdot (aF_W(u))^2) &= \Tr_{q/p}(F_W(u) \cdot aF_W(u))
=\Tr_{q/p}(u  F_W^\ast(aF_W(u))) \\
&= \Tr_{q/p}(u(R+R^\ast)(u))=0. 
\end{align*}
Moreover, for every $v\in W=\ker F_W^\ast$, 
\[
\Tr_{q/p}(a^{-1} \cdot (aF_W(u)v))=\Tr_{q/p}(uF_W^\ast(v))=0. 
\]
This shows that
\[
\Tr_{q/p}(a^{-1}w^2)=0
\qquad\text{for all }w\in W'.
\]

Hence $\overline{D}_{\rho'}$ is $\F_q$-extremal by
Theorem~\ref{t1}(i). Moreover, since $a$ is unchanged, the criterion in
Theorem~\ref{t1}(i) shows that $\overline{D}_{\rho'}$ is $\F_q$-maximal
(resp.\ $\F_q$-minimal) if and only if $\overline{D}_{\rho}$ is
$\F_q$-maximal (resp.\ $\F_q$-minimal). 

\medskip

It remains to show that $D_{\rho}$ is a quotient of $D_{\rho'}$. Since $W\subset W'$, we may choose an $\F_p$-linearized polynomial 
$h$ such that
\[
F_{W'}=F_W\circ h.
\]
By Lemma~\ref{lem: Fast aF=R+Rast}(ii), we have 
\[
2xR(x)=2xR_\rho(x)\sim aF_W(x)^2. 
\]
Substituting $h(x)$ for $x$ and using $F_{W'}=F_W\circ h$, we obtain 
\begin{equation}\label{hRhR'}
    2h(x)R_\rho(h(x))\sim aF_W(h(x))^2=aF_{W'}(x)^2\sim 2x R_{\rho'}(x). 
\end{equation}

By~\eqref{hRhR'} and Lemma \ref{abc}, we obtain  a finite \'etale surjective morphism 
\[
D_{\rho'} \longrightarrow D_\rho. 
\]
Hence $D_\rho$ is a quotient of $D_{\rho'}$, as claimed. 

\subsubsection{The case where $p$ is even}\label{import}
In this case, by Theorem~\ref{recipe for ex curve intro}(ii), we may assume that
$C_R\cong D_\epsilon$, where $\epsilon=(W,t)$ consists of an
$\F_p$-subspace $W\subset \F_q$ containing $1$ and an element
$t\in\F_q$ satisfying
\[
Q_q(v)=\psi_q(tv)\qquad\text{for all }v\in W.
\]
Assume that $R=R_\epsilon$.
Then
\[
R+R^\ast=F_W^\ast\circ F_W.
\]

By Lemma~\ref{lemma: existence of u}, we may choose an element $u\in A\setminus V_R$. 
Set \[
W'\coloneqq W+\langle{F_W(u)}\rangle_{\F_p}.
\]
Since $u\notin V_R=\ker (F_W^\ast \circ F_W)$ and 
$W=\ker F_W^\ast$, we have 
$F_W(u) \not\in W$. Hence 
$\dim_{\F_p} W'=e+1$.

Set $\epsilon':=(W',t)$. We show that $\overline{D}_{\epsilon'}$ is $\F_q$-extremal by verifying the condition in
Theorem~\ref{recipe for ex curve intro}(i). We need the following lemma. 
\begin{lemma}\label{lemma: F,tF sim  xRx}
The following statements hold. 
\begin{enumerate}
    \item We use the notation in Lemma~\ref{lem: Fast F=R+Rast}. In the ring $W_2(\F_q[x])$, we have 
    \[
    \left(F_W(x),(tF_W(x))^2\right)\sim (0,xR(x)). 
    \]
    \item For every $x\in \F_q$, we have 
    \[
Q_q(F_W(x))\psi_q(tF_W(x))^{-1}=\psi_q(xR(x)). 
    \]
\end{enumerate}
\end{lemma}
\begin{proof}
    (i) By Lemma~\ref{lem: Fast F=R+Rast}(ii), we have 
\[
(F_W(x),0)\sim\bigl(0,x(R(x)-F_W^\ast(t)^2x)\bigr). 
\]

Since $(0,(F_W^\ast(t) x)^2) \sim (0,(t F_W(x))^2)$, the assertion follows. 

\noindent
(ii) By~(i), we have 
\[
Q_q(F_W(x))\psi_q((tF_W(x))^2)=\psi_q(xR(x)).
\]
Since $\psi_q((tF_W(x))^2)=\psi_q(tF_W(x))^{-1}$, the assertion follows. 
\end{proof}

We show that 
\begin{equation}\label{equation: Qqw}
Q_q(w)\psi_q(tw)^{-1}=1\qquad\text{for all }w\in W'. 
\end{equation}
Since $W'=F_W(V_R+\langle u\rangle_{\F_p})$, we may write $w=F_W(x+au)$ with $x\in V_R$ and $a\in \F_p$. Then, by Lemma~\ref{lemma: F,tF sim  xRx}(ii), 
\begin{align*} Q_q(w)\psi_q(tw)^{-1}&=\psi_q((x+au)R(x+au))\\&=\psi_q(xR(x))\cdot\psi_q(xR(au)+auR(x))\cdot\psi_q(a^2uR(u)).
\end{align*}
We claim that each factor on the right-hand side is equal to $1$. 
Since $F_W(x)\in W$, we have 
\[
Q_q(F_W(x))=\psi_q(tF_W(x)). 
\]
By Lemma~\ref{lemma: F,tF sim  xRx}(ii), this implies that  $\psi_q(xR(x))=1$.

For the second factor, by $x R(au) \sim au R^\ast(x)$,  we have
\[
\psi_q(xR(au)+auR(x))
=
\psi_q\bigl(au(R(x)+R^\ast(x))\bigr)
=
1.
\]
Finally, since $a \in \F_p$ and $u\in A$, we have $\psi_q(a^2uR(u))=1$.

Consequently, \eqref{equation: Qqw} holds, and hence $\overline{D}_{\epsilon'}$ is $\F_q$-extremal. Moreover, since $t$ is unchanged, the criterion in
Theorem~\ref{recipe for ex curve intro}(i) shows that $\overline{D}_{\epsilon'}$ is $\F_q$-maximal
(resp.\ $\F_q$-minimal) if and only if $\overline{D}_{\epsilon}$ is
$\F_q$-maximal (resp.\ $\F_q$-minimal). 

\medskip

It remains to show that $D_{\epsilon}$ is a quotient of $D_{\epsilon'}$.  Since $W\subset W'$, there exists an $\F_p$-linearized polynomial $h\in\F_q[x]$ such that 
\[F_{W'}=F_W\circ h. \]
Substituting $h(x)$ for $x$ in Lemma~\ref{lemma: F,tF sim  xRx}(i), we obtain 
\[
(0,h(x)R(h(x)))\sim \left(F_{W'}(x),(tF_{W'}(x))^2\right)\sim (0,xR_{\epsilon'}(x)). 
\]
Hence, there exists $(c,d)\in W_2(\F_q[x])$ satisfying 
\[
(0,h(x)R(h(x))-xR_{\epsilon'}(x))=(c^p,d^p)-(c,d)=(c^p-c,\, d^p-d-c(c^p-c)). 
\]
This shows that $c\in \F_p$ and 
\[
h(x)R(h(x)) - xR_{\epsilon'}(x)=d^p-d. 
\]
By Lemma~\ref{abc}, 
there exists a finite \'etale surjective morphism 
\[
D_{\epsilon'} \longrightarrow
D_{\epsilon}.
\]
This completes the proof of Theorem~\ref{exist12}.

\begin{remark}\label{remark: e>n does not occur}
Let 
\[
R(x)=\sum_{i=0}^ea_ix^{p^i}\in\F_q[x]
\]
  be an $\F_p$-linearized polynomial with $a_e\neq0$. It is well known that $\overline{C}_R$ cannot be $\F_q$-extremal
if $2e>[\F_q:\F_p]$. 

In this remark, we give a proof of this fact using the results in Section~3. We treat the case where $p_0$ is odd; the case $p_0=2$ is similar. 

Assume that $\overline{C}_R$ is $\F_q$-extremal. Then, by Theorem~\ref{t1}(ii), there exists a pair $\rho=(W,a)$ such that $R=R_\rho$. Therefore,
\[
p^e=\deg R_\rho=\#W.
\]
On the other hand, by Theorem~\ref{t1}(i), $W$ is a subspace of $\F_q$ that is totally isotropic with respect to the bilinear form $\Tr_{q/p}(a^{-1}xy)$. Since this bilinear form is nondegenerate, we must have 
\[
\dim_{\F_p}W\le \frac{1}{2}[\F_q:\F_p],\]
and hence $2e\le [\F_q:\F_p]$. 
\end{remark}

\subsection{Existence of maximal and minimal curves}

The following theorem was proved in \cite[Theorem~5.9]{CO0}. 

\begin{theorem}\label{theorem: existence of extremal curves}
Assume that $n=[\F_q:\F_p]$ is even. 
Then the following statements hold. 
\begin{enumerate}
    \item Let $e$ be an integer satisfying 
    \[
\begin{cases}
0 \le e \le n/2 & \textrm{if $p$ is odd}, \\
1 \le e \le n/2 &\textrm{if $p$ is even}. 
\end{cases}
\]
Then there exists an $\F_p$-linearized polynomial $R(x) \in \F_q[x]$ of degree $p^e$ such that the curve $\overline{C}_R$ is $\F_q$-maximal. 
    \item Let $e$ be an integer satisfying 
    \[
\begin{cases}
0 \le e \le (n/2)-1 & \textrm{if $p$ is odd}, \\
1 \le e \le (n/2)-1 &\textrm{if $p$ is even}. 
\end{cases}
\]
Then there exists an $\F_p$-linearized polynomial $R(x) \in \F_q[x]$ of degree $p^e$ such that the curve $\overline{C}_R$ is $\F_q$-minimal. 
\end{enumerate}
\end{theorem}

By Theorem~\ref{exist12}, the above theorem follows from the following lemma. 
\begin{lemma}\label{key}
The following statements hold. 
\begin{itemize}
\item[{\rm (i)}] 
Assume that $p$ is odd. 
Let $R(x)=ax$. Then $\overline{C}_{R}$ is 
$\F_q$-maximal (resp.\ $\F_q$-minimal) if and only if
\[
\left(\frac{a}{q}\right) \left(\frac{-1}{\sqrt{q}}\right)=-1 \quad (resp.\ \left(\frac{a}{q}\right) \left(\frac{-1}{\sqrt{q}}\right)=1). 
\]
\item[{\rm (ii)}] Assume that $p$ is even. 
Let $R_a(x)=x^p+ax$ for $a \in \F_q$. Then the family of curves 
$\{\overline{C}_{R_a}\}_{a \in \F_q}$ contains an $\F_q$-maximal curve. 
Furthermore, if $[\F_q:\F_p]>2$, then this family also contains an $\F_q$-minimal curve. 
\end{itemize}
\end{lemma}
\begin{proof}
(i) The assertion follows from 
Theorem~\ref{t1} applied 
to $(W,a)=(\{0\},a)$. 

\noindent
(ii) Set $W := \F_p$, and 
let $t \in \F_q$ satisfy $t^{\sqrt{q}} + t = 1$. Set $\epsilon:=(W,t)$. 

Since $W\subset\F_{\sqrt{q}}$, for every $v\in W$, we have
\[
Q_q(v)=\psi_{\sqrt{q}}(v)=\psi_q(tv). 
\]
By Theorem~\ref{recipe for ex curve intro}(i), the curve $\overline{D}_{\epsilon}$ is $\F_q$-extremal. 
Let $m\coloneqq [\F_q:\F_p]/2$. 
A direct computation shows that 
\begin{align*}
Q_q(t)
&=Q_{\sqrt{q}}(t^{\sqrt{q}}+t,t^{\sqrt{q}+1})\\
&=Q_{\sqrt{q}}(1,t+t^2)\\
&=\xi_2\!\left(m(1,0)+(0,\Tr_{\sqrt{q}/2}(t^2+t))\right)\\
&=(\sqrt{-1})^{m}\cdot \xi_2(0,t^{\sqrt{q}}+t)\\
&=(\sqrt{-1})^{m}\cdot \xi_2(0,1)\\
&=-(\sqrt{-1})^{m}.
\end{align*}
Hence, by Theorem~\ref{recipe for ex curve intro}(i),
$\overline{D}_{\epsilon}$ is $\F_q$-maximal.

Assume now that $m>1$. 
We show that there exists $u \in \F_q$ such that 
\begin{equation}\label{qqt}
Q_q(t + F_W(u)) = - Q_q(t).
\end{equation}

Since $Q_q(x + y) = Q_q(x)\, Q_q(y)\, \psi_q(xy)$ for $x,y \in \F_q$, we compute 
\begin{align*}
Q_q(t + F_W(u))Q_q(t)^{-1}
&= Q_q(F_W(u))\, \psi_q(tF_W(u)) \\
&=  \psi_q(uR_{\epsilon}(u)), 
\end{align*}
where the last equality follows from Lemma~\ref{lemma: F,tF sim  xRx}(ii). 

Assume that $\psi_q\bigl(u R_{\epsilon}(u)\bigr) = 1$ for all $u \in \F_q$. Then, for every $a\in\F_p$, 
\[
\psi(a\Tr_{q/p}(u R_{\epsilon}(u)))=
\psi_q\bigl(au R_{\epsilon}(u)\bigr) =\psi_q\bigl(a^{1/2}u R_{\epsilon}(a^{1/2}u)\bigr) = 1. 
\]
Hence, 
\[
\Tr_{q/p}\bigl(u R_{\epsilon}(u)\bigr) = 0
\quad \text{for all } u \in \F_q.
\]
By Artin--Schreier theory, it follows that 
\[
\# D_{\epsilon}(\F_q) =\# \left\{(u,v) \in \F_q^2 \mid 
v^p-v=u R_{\epsilon}(u)\right\}=pq. 
\]
On the other hand, since $\overline{D}_\epsilon$ is $\F_q$-maximal and has genus $p(p-1)/2$, we have 
\[
\# D_{\epsilon}(\F_q) = q + p(p-1)p^m. 
\]
Since $m > 1$, we have $q + p(p-1)p^m\neq pq$, a contradiction. 

Therefore, there exists $u \in \F_q$ such that 
\[
Q_q(t + F_W(u)) = - Q_q(t).
\]
Using \eqref{qqt}, for $v \in W=\ker F_W^\ast$, 
\begin{align*}
\psi_q((t+F_W(u))v)&=\psi_q(tv) \psi_q(F_W(u) v)\\
&=\psi_q(tv)\psi_q(u F_W^\ast(v))=\psi_q(tv)=Q_q(v). 
\end{align*}
By Theorem~\ref{recipe for ex curve intro}(i), the curve $\overline{D}_{(W,t+F_W(u))}$ is $\F_q$-minimal. 
\end{proof}

\subsubsection{Quadratic forms and extremal curves}
Assume that $p$ is odd. By Theorem~\ref{t1}, totally isotropic
subspaces for the quadratic form
\[
x\mapsto \Tr_{q/p}(a^{-1}x^2)
\]
give rise to $\F_q$-extremal curves.  In this subsection, we give an alternative proof of Theorem~\ref{theorem: existence of extremal curves} from this viewpoint. 

Recall that a quadratic space over $\F_p$ is a finite-dimensional
$\F_p$-vector space $V$ equipped with a quadratic form
\[
Q:V\to\F_p.
\]
The associated symmetric bilinear form is 
\[B_Q(x,y):=2^{-1}(Q(x+y)-Q(x)-Q(y)). \]
We say that $Q$ is nondegenerate if 
 $B_Q$ is nondegenerate. 
In this case, 
the \emph{discriminant} of $Q$ is defined by 
\[
\disc(Q):=(-1)^{n(n-1)/2} \cdot \det B_Q \in \F_p^{\times}/(\F_p^{\times})^2,  
\]
where $n=\dim_{\F_p}V$ and $\det B_Q \in \F_p^{\times}/(\F_p^{\times})^2$ denotes the determinant of $B_Q$.
\begin{proposition}\label{split}
Assume that $n=[\F_q:\F_p]$ is even. 
Let $a \in \F_q^{\times}$ and let 
\[Q_a \colon \F_q \to \F_p,\ x \mapsto \Tr_{q/p}(a^{-1}x^2).\]
This is a nondegenerate quadratic form on $\F_q$. 
Then 
\[
\left(\frac{\disc(Q_a)}{p}\right)=-\left(\frac{a}{q}\right)\left(\frac{-1}{\sqrt{q}}\right). 
\]
The quadratic space $(\F_q,Q_a)$ is split if and only if 
$\disc(Q_a)=1$. 
\end{proposition}
\begin{proof}
Let $\{\alpha,\ldots,\alpha^{p^{n-1}}\}$
be a normal basis of $\F_q/\F_p$. 
Then the discriminant of the field extension $\F_q/\F_p$ is 
\[
D^2\in \F_p^{\times}/(\F_p^{\times})^2, \qquad 
\textrm{where} \quad D:=\prod_{0 \le i<j \le n-1} 
(\alpha^{p^i}-\alpha^{p^j}). 
\]
By assumption, $n$ is even. 
Since $D^p=(-1)^{n-1}D=-D$, 
we have $\bigl(\frac{D^2}{p}\bigr)
=D^{p-1}=-1$. 
The discriminant of $Q_1$ is given by 
\[
\left(\frac{\disc(Q_1)}{p}\right)=\left(\frac{(-1)^{n(n-1)/2} D^2}{p}\right)=-\left(\frac{(-1)^{n/2}}{p}\right)=-\left(\frac{-1}{\sqrt{q}}\right). 
\]
Using the norm map $\Nr_{q/p} \colon \F_q^{\times} \to \F_p^{\times}$, we obtain
\[
\left(\frac{\disc(Q_a)}{p}\right)=
\left(\frac{\Nr_{q/p}(a)^{-1}\disc(Q_1)}{p}\right)
=-\left(\frac{a}{q}\right)\left(\frac{-1}{\sqrt{q}}\right). 
\]
The final assertion follows from the classification of
nondegenerate quadratic spaces over finite fields (see, e.g., \cite[Chapter~II, \S3, p.~30 and p.~36]{Lam}). 
\end{proof}
\begin{corollary}\label{exist}
Assume that $p$ is odd and that $n=[\F_q:\F_p]$ is even. 
\begin{itemize}
\item[{\rm (i)}] For any integer $0 \le e \le n/2$, 
there exists an $\F_q$-maximal curve in the family of van der Geer--van der Vlugt curves 
of genus $p^e(p-1)/2$.
\item[{\rm (ii)}] For any integer $0 \le e \le (n/2)-1$, 
there exists an $\F_q$-minimal curve in the same 
family 
of genus $p^e(p-1)/2$.
\end{itemize}
\end{corollary}
\begin{proof}
Choose $a \in \F_q^{\times}$ such that 
\[
\left(\frac{a}{q}\right)=-\left(\frac{-1}{\sqrt{q}}\right)
 \qquad (\textrm{resp.}\ \left(\frac{a}{q}\right)=\left(\frac{-1}{\sqrt{q}}\right)).
\]
We consider the quadratic space $(\F_q,Q_a)$, 
where $Q_a \colon \F_q \to \F_p,\ x \mapsto \Tr_{q/p}(a^{-1} x^2)$. 
Choose a maximal totally isotropic subspace $V_a \subset \F_q$
with respect to $Q_a$. 
By Lemma~\ref{split}, the quadratic space
$(\F_q,Q_a)$ is split (resp.\ non-split).
Hence its Witt index is
$n/2$ (resp.\ $(n/2)-1$). 
Hence we can take an $\F_p$-vector subspace $W \subset V_a$ of dimension $e$. 
From Theorem \ref{t1}(i), 
it follows that the curve $\overline{D}_{(W,a)}$ is 
$\F_q$-maximal (resp.\ $\F_q$-minimal). This completes the proof. 
\end{proof}

\subsection{Maximal curves are quotients of the Hermitian curve}\label{appB}
Let $p$ be a power of a prime number $p_0$, and let $q$ be a power of $p$. Assume that $n=[\F_q:\F_p]$ is even. Let 
\[
R(x)=\sum_{i=0}^ea_ix^{p^i}\in \F_q[x]
\]
be an $\F_p$-linearized polynomial with $a_e\neq0$ such that the curve $\overline{C}_R$ is $\F_q$-maximal. 
It was proved in \cite[Theorem 6.12]{CO0} that such curves are quotients of the smooth affine curve 
\[
H\colon y^{p^{n/2}}+y=x^{p^{n/2}+1}, 
\]
whose compactification is isomorphic to the Hermitian curve. 

The Hermitian curve is known to be $\F_q$-maximal. Therefore, every quotient of the Hermitian curve is also $\F_q$-maximal. The above theorem \cite[Theorem 6.12]{CO0} provides an interesting converse for van der Geer--van der Vlugt curves.

In this subsection, we give an alternative proof of the following more precise result using Theorem~\ref{exist12}(i). 
\begin{theorem}\label{theorem: quotient of Hermitian curve}
Let 
\[
R(x)=\sum_{i=0}^ea_ix^{p^i}\in \F_q[x]
\]
be an $\F_p$-linearized polynomial with $a_e\neq0$. Assume that $\overline{C}_R$ is $\F_q$-maximal. Then there exists a finite \'etale surjective morphism 
\[
H\to C_R. 
\] 
\end{theorem}

Assume that $\overline{C}_R$ is $\F_q$-maximal.  Then, by Remark~\ref{remark: e>n does not occur}, we have $e\le n/2$. Therefore, by Theorem~\ref{exist12}(i), there exists a finite \'etale surjective morphism 
\[
C_{R_1}\to C_R
\]
where $R_1\in \F_q[x]$ is an $\F_p$-linearized polynomial of degree $p^{n/2}$ and $\overline{C}_{R_1}$ is $\F_q$-maximal. 

Thus it suffices to prove the theorem in the case $e=n/2$.
In this case, Theorem~\ref{theorem: quotient of Hermitian curve}
follows from the following two lemmas. 

\begin{lemma}\label{app1}
Assume that $\overline{C}_R$ is $\F_q$-maximal and that  
$e=n/2$. Then $R(x)=ax^{p^e}$ for some $a\in\F_q^\times$
satisfying $a^{p^e-1}=-1$. 
\end{lemma}
\begin{proof}

Define 
\[
E_R:=(R+R^\ast)^{p^e},\qquad V_R:=\ker E_R. 
\]
Since $\overline{C}_R$ is $\F_q$-extremal, we have 
$V_R \subset \F_q$ by Lemma~\ref{lemma: When CR is extremal Fq Fp is even and VR in Fq}. 
Moreover,
\[
\dim_{\F_p}V_R=2e=n=\dim_{\F_p}\F_q.
\]
Hence $V_R=\F_q$.

We have 
\[
E_R(x)=\sum_{i=0}^{n/2} a_i^{p^e}x^{p^{i+e}}+\sum_{i=0}^{n/2}
a_i^{p^{e-i}}x^{p^{e-i}}. 
\]
 Since $\F_q\subset V_R$, the polynomial $E_R(x)$ is divisible by $x^q-x$. Since $\deg E_R=q$, it follows that
\[
E_R(x)=\xi(x^q-x)
\]
for some $\xi\in\F_q^\times$.
 Comparing coefficients, we obtain 
\[
2a_0=0, \qquad 
a_1=\cdots=a_{e-1}=0, \qquad 
a_{e}^{p^e}=\xi=-a_e. 
\]
This shows that $R(x)=a_e x^{p^e}+a_0 x$ with $a_e^{p^e-1}=-1$. When $p$ is odd, we further have $a_0=0$, and the claim follows.  

Assume now that $p$ is even. In this case, $a_e^{p^e-1}=-1=1$, and hence $a_e\in\F_{p^e}$. 
Since $\overline{C}_R$ is $\F_q$-maximal, we have 
\[
\# C_R(\F_q)=q+p^e(p-1)\sqrt{q}. 
\]
Since $e=n/2$, we obtain 
\[
\# C_R(\F_q)=pq. 
\]
Thus
\[
\#\left\{x\in \F_q\mid\Tr_{q/p}(xR(x))=0\right\}=q, 
\]
and hence, 
\[
\Tr_{q/p}(xR(x))=0
\qquad\text{for all }x\in\F_q.
\]
Since $a_e x^{p^e+1} \in \F_{p^e}$ and $[\F_q:\F_{p^e}]=2$, we have 
$\Tr_{q/p}(a_ex^{p^e+1})=0$. Therefore, 
\[\Tr_{q/p}(a_0 x^2)=\Tr_{q/p}(x R(x))=0. \]
Since $x\mapsto x^2$ is bijective on $\F_q$, it follows that $\Tr_{q/p}(a_0 x)=0$ for all $x\in \F_q$. 
Since the pairing
\[
\F_q\times\F_q\to\F_p,\qquad (x,y)\mapsto\Tr_{q/p}(xy)
\]
is nondegenerate, we obtain $a_0=0$.  Hence the claim follows. 
\end{proof}
\begin{lemma}\label{app2}
Assume that $R(x)=ax^{p^e}$ where $a\in \F_q^\times$ satisfies $a^{p^e-1}=-1$. Moreover, assume that $e=n/2$. Then there exists a finite \'etale surjective morphism 
\[
H\to C_R. 
\]
\end{lemma}
\begin{proof}
A direct computation shows that the morphism
\[
H \longrightarrow C_R,
\qquad
(x,y)\longmapsto
\left(
x,\,
\sum_{i=0}^{n-1}(-ay)^{p^i}
\right)
\]
is finite, \'etale, and surjective.
\end{proof}

This completes the proof of Theorem~\ref{theorem: quotient of Hermitian curve}.

\subsection*{Acknowledgement} 
T. I. is supported by JSPS KAKENHI Grant Numbers 23K20786, 24K21512 and 25K00905.\\
D. T. is supported by JSPS KAKENHI Grant Number 25KJ0122.\\
T. T. is supported by JSPS KAKENHI Grant Numbers 25K06959 and 23K20786.

\end{document}